\documentclass[1p]{elsarticle}

\usepackage{amssymb}
\usepackage{amscd}
\usepackage{amsmath}
\usepackage{natbib}

\newtheorem{thm}{Theorem}

\newtheorem{prop}[thm]{Proposition}
\newtheorem{lem}[thm]{Lemma}
\newtheorem{rem}[thm]{Remark}
\newtheorem{conjecture}[thm]{Conjecture}
\newproof{proof}{Proof}

\def\R{\mathbb R}
\def\C{\mathbb C}
\def\Q{\mathbb Q}
\def\Z{\mathbb Z}
\def\F{\mathbb F}
\def\Qbar{\overline{\mathbb {Q}_p}}

\DeclareMathOperator{\gal}{Gal}
\DeclareMathOperator{\pr}{pr}
\DeclareMathOperator{\GL}{GL}

\begin{document}

\title{Valuations of $p$-adic regulators of cyclic cubic fields}

\author[tukl]{Tommy Hofmann\corref{cor1}}
\address[tukl]{Fachbereich Mathematik, Technische Universit\"at Kaiserslautern, Postfach 3049, 67653 Kaiserslautern, Germany}
\ead{thofmann@mathematik.uni-kl.de}

\cortext[cor1]{Corresponding author}

\author[usyd]{Yinan Zhang}
\address[usyd]{School of Mathematics and Statistics, The University of Sydney}
\ead{y.zhang@sydney.edu.au}

\begin{abstract}
We compute the $p$-adic regulator of cyclic cubic extensions of $\Q$ with discriminant up to $10^{16}$ for $3<p<100$, and observe the distribution of the $p$-adic valuation of the regulators. We find that for almost all primes, the observation matches the model that the entries in the regulator matrix are random elements with respect to the obvious restrictions.
Based on this random matrix model, a conjecture on the distribution of the valuations of $p$-adic regulators of cyclic cubic fields is stated.
\end{abstract}



\maketitle

\section{Introduction}

The class group of a number field $K$ is an important invariant of the field, providing information on the multiplicative structure of the number field. Another invariant of the field, the regulator, provides information on the unit group structure, and is intimately linked to the the class group by the class number formula. Various algorithms for class group computation have emerged over the years, culminating in Buchmann's subexponential algorithm which can be used compute class groups of arbitrary number fields \cite{Buc}. Further improvements by Cohen, Diaz y Diaz and Olivier \cite{CDO} allowed computation of both the class group and regulator in the same algorithm. Despite this and further improvements to the algorithm in the last 25 years \cite{BDF,BF}, it is still impractical to compute the class group and regulator of number fields with large discriminants. The ability to compute class groups and regulators efficiently for arbitrary number fields has, and remains, a key focus in computational number theory.

With the development of these algorithms and the availability of computational resources, one has seen a rise of what is nowadays called arithmetic statistics. Instead of considering invariants of single arithmetic objects, one investigates the statistics of invariants of families of arithmetic objects. This has led to various deep conjectures about invariants of number fields, most notably Malle's conjecture on the distribution of Galois groups \cite{Malle2004}, the Cohen--Lenstra--Martinet conjecture on the distribution of class groups \cite{Cohen1984, Cohen1987} and Malle's refinement thereof \cite{Malle2010}. We extend this list of conjectures by providing a heuristic on the distribution of $p$-adic regulators of cyclic cubic fields.

The $p$-adic regulator $R_p(K)$ of a number field $K$ was introduced by Leopoldt in his investigation of $p$-adic $L$-functions, and he conjectured that it is non vanishing. Unlike its classical counterpart, the $p$-adic regulator is only well defined if the number field is totally real abelian or CM. As a result, very little information is known about the $p$-adic regulator, in contrast to the classical regulator. Previous research on computing $p$-adic regulators has been predominantly focused on numerical verification of Leopoldt's conjecture rather than computing their exact value.

Limited work on the valuation of $p$-adic regulator was carried out by Miki \cite{Hi}, who attempted to provide an upper bound on $v_p(R_p(K))$. This did not turn out to be useful in practice, as the formula contained terms whose values are not explicit. A simple lower bound on the valuation was also briefly mentioned by Hakkarainen in his PhD thesis \cite{H}, but overall there has been little focus on this area, due to mainly difficulties with computing $R_p(K)$. In particular, Panayi, who was one of the first to compute $R_p(K)$ explicitly in his PhD thesis \cite{P}, noted that there were significant practical difficulties with computing in $p$-adic fields.

Recent development in a $p$-adic class number algorithm by Fieker and Zhang \cite{FZ} for totally real abelian fields made it possible to compute the $p$-adic regulator for these fields in a relatively efficient manner. This allowed us to compute the $p$-adic regulator for a large number of cyclic cubic extensions of $\Q$, and the experimental data allowed us to conjecture and provide heuristics on the distribution of the values of $v_p(R_p(K))$.

Let $\mathcal K$ be the set of all cyclic cubic extensions of $\Q$ inside a fixed algebraic closure of $\Q$. Note that such extensions are necessarily totally real. For a prime $p$ let $\mathcal K_p^{\mathrm{un}}$ and $\mathcal K_p^{\mathrm{ram}}$ respectively denote the set of all fields in $\mathcal K$ which are unramified and ramified at $p$ respectively.
Note that $\mathcal K_p^{\mathrm{ram}} = \emptyset$ and $\mathcal K = \mathcal K_p^{\mathrm{un}}$ in case $p \equiv 2 \bmod 3$.
For $D > 0$ we set $\mathcal K(D) = \{ K \in \mathcal K \mid \lvert d(K) \rvert \leq D \}$, where $d(K)$ is the discriminant of $K$.
Moreover we set $\mathcal K_p^{\mathrm{un}}(D) = \mathcal K_p^{\mathrm{un}} \cap \mathcal K(D)$ and $\mathcal K_p^{\mathrm{ram}}(D) = \mathcal K_p^{\mathrm{ram}} \cap \mathcal K(D)$.
Based on heuristics and numerical data, we make the following conjecture:

\begin{conjecture}\label{conj}
  For primes $p > 3$ the following hold:
  \begin{enumerate}[(i)]
  \item
  If $p \equiv 2 \bmod 3$, then $v_p(R_p(K)) \in 2 \Z$ for all $K \in \mathcal K$ and for $i \geq 0$ we have
  \[ \lim_{D \to \infty} \frac{\#\{ K \in \mathcal K(D) \mid v_p(R_p(K)) = 2i \}}{\# \mathcal K(D)} = \frac 1 {p^{2i - 2}} \left( 1- \frac  1 {p^2} \right). \]
  \item
  If $p \equiv 1 \bmod 3$, then for $i \geq 0$ we have
  \[ \lim_{D \to \infty} \frac{\#\{ K \in \mathcal K_p^{\mathrm{un}}(D) \mid v_p(R_p(K)) = i \}}{\#\mathcal K_p^{\mathrm{un}}(D)} = \frac{i-1}{p^{i-2}}\left( 1 - \frac 1 p \right)^2. \]
  and
  \[ \lim_{D \to \infty} \frac{\#\{ K \in \mathcal K_p^{\mathrm{ram}}(D) \mid v_p(R_p(K)) = i\}}{\# \mathcal K_p^{\mathrm{ram}}(D)} =  \frac{i}{p^{i-1}}\left( 1 - \frac 1 p \right)^2.  \]
  \end{enumerate}
\end{conjecture}

\section{Definition and algorithm}

Let $K$ be a number field of degree $n$ with unit group $U_K$.
We first define the (classical) regulator $R(K)$ and its $p$-adic analogue $R_p(K)$, as the definition of $R_p(K)$ can be ambiguous.
The field $K$ has $r_1$ embeddings into $\R$, denoted as $\tau_1,\dots,\tau_{r_1}$, and $2r_2$ embeddings into $\C$, denoted as conjugate pairs $\tau_{{r_1}+1},\overline{\tau}_{{r_1}+1},\dotsc,\tau_{{r_1}+{r_2}}, \overline{\tau}_{{r_1}+{r_2}}$, with $r_1+2r_2=n$.
Let $u_1,\dotsc,u_{{r_1}+{r_2}-1}$ be a system of fundamental units of $U_K$, that is, these elements form a basis of the torsion-free quotient of $U_K$. Consider the submatrix formed by deleting one column of the matrix
\[(\delta_i \log \rvert \tau_i(u_j)\rvert)_{i,j} \in \operatorname{Mat}_{r_1 + r_2 - 1}(\R) \,,\]
where $\delta_i=1$ for $1\leq i \leq r_1$ and $2$ if $r_1+1\leq i \leq r_1+r_2$. As each row sums to zero, the determinant of such a submatrix is independent of the column deleted. The absolute value of this determinant, denoted by $R(K)$, is also independent of the choice of the fundamental units and is known as the regulator of the number field $K$.

\subsection{The $p$-adic regulator}
The $p$-adic regulator can be defined similarly, with changes to the embeddings. By fixing an embedding from $\C_p$ into $\C$, any embedding of $K$ into $\C_p$ can be considered as either real or complex, depending on the image of $K$ in the composite embedding into $\C$. For totally real or CM fields, whether an embedding from $K$ to $\C_p$ is real or complex is independent of the choice of embedding from $\C_p$ to $\C$, but ambiguities will arise in other number fields.
As cyclic cubic fields are totally real, the $p$-adic regulator is a well-defined object and does not depend on any choice.

We denote by $\sigma_1,\dots,\sigma_{r_1}$ the real embeddings and by $\sigma_{{r_1}+1},\overline{\sigma}_{{r_1}+1},\dots$, $\sigma_{{r_1}+{r_2}}$, $\overline{\sigma}_{{r_1}+{r_2}}$ the complex embeddings from $K$ to $\C_p$. Then, the $p$-adic regulator $R_p(K)$ is defined as the determinant of the submatrix obtained by removing one column of the matrix
\[(\delta_i \log_p\sigma_i(u_j))_{i,j} \in \operatorname{Mat}_{r_1 + r_2 - 1}(\C_p) \,,\]
where $\log_p$ is the $p$-adic logarithm, and $u_1,\dots,u_{{r_1}+{r_2}-1}$ and $\delta_i$ as defined previously.
As for the ordinary regulator, this value is up to units independent of the choice of the fundamental units and the column deleted.

Alternatively instead of deleting a column in the matrix, one can add a row of $1$'s to it. Again, due to each row summing to zero, the determinant is unaffected. This definition was introduced by Iwasawa \cite{I} and was subsequently implemented in the algorithm by Fieker and Zhang \cite{FZ}, part of which is used here with some modifications. This does have the disadvantage of calculating the determinant of a matrix one dimension higher than necessary, but it turned out that this is outweighed by leaving the structure of the original matrix intact.

It is important to note that we can replace the fundamental units with a set of independent units which generate a $p$-maximal subgroup of $U_K$.
More precisely, let $u_1',\dotsc,u_{r_1+r_2-1}'$ be multiplicatively independent units and assume that the index of $\langle u_1',\dotsc,u_{r+1+r_2-1}', \mu_K \rangle$ in $U_K$ is coprime to $p$ (here $\mu_K$ is the set of roots of unity of $K$).
Then one can show that the determinant of the submatrix obtained by removing one column of the matrix
\[ (\delta_i \log_p\sigma_i(u'_j))_{i,j} \in \operatorname{Mat}_{r_1 + r_2 - 1}(\C_p) \,,\]
is equal to the $p$-adic regulator modulo $p$-adic units.

By applying saturation techniques of Biasse and Fieker \cite{BF} on a tentative unit group, we can obtain a $p$-maximal subgroup of $U_K$. The units are represented as a power product of some small elements $\alpha_j$ and some large exponents $e_{i,j}\in\Z$ in the form of 
\[u_i = \prod_{j=1}^r \alpha_j^{e_{i,j}}\,,\]
instead of in terms of some fixed basis of the field. This is due to the fact that $u_i$ may require up to $O(\sqrt{|d(K)|}/n)$ digits, and the power product representation not only permits more efficient storage of the result, but also speeds up the actual computation of $\log_p\sigma_j(u_i)$. The different $p$-adic embeddings are computed using standard techniques for $p$-adic factorisation or root finding. Alternatively once we have one fixed $p$-adic embedding we can make use of $\Q$-automorphisms of the field to find the remaining embeddings.

\subsection{Tables of number fields}
In order to test our conjecture, we were in need of a large number of cyclic cubic extensions of $\Q$.  One possibility is to use the tables of cyclic cubic fields of discriminant $< 10^{16}$ compiled by Malle, which were used in \cite{Malle2008} and based on the algorithm described in \cite{Cohen1993}.
Instead, we took the opportunity to validate the data by computing the same table using a different method.
More precisely, we used the algorithms based on global class field theory as provided by Fieker in \cite{Fieker2001} (see also \cite{Cohen2000}).
Since the base field is just $\Q$, theory tells us that the cyclic cubic extensions $K|\Q$ are in bijection with pairs $(f, U)$, where $f \in \Z_{>0}$ is an integer and $U \subseteq (\Z/f\Z)^\times$ is a subgroup of index $3$, modulo an appropriate equivalence relation.
In fact $f$ can be taken to be the conductor of $K$, which is the smallest integer such that $K \subseteq \Q(\zeta_f)$. 
Hence in order to list all cyclic cubic fields of bounded conductor we enumerate all pairs $(f, U)$ with bounded $f$.
Since the discriminant $d_K$ of a cyclic cubic field $K$ of conductor $f$ satisfies $d_k = f^2$, in this way we can enumerate all cyclic cubic fields with bounded discriminant.
To obtain a defining equation of $K|\Q$ given the pair $(f, U)$ requires more work and involves the computing of discrete logarithms in $(\Z/f \Z)^\times$ (see the references for more details).
Since the conductor $f$ must satisfy $\# (\Z/f\Z)^\times \equiv 0 \bmod 3$ and the prime divisors of $f$ are exactly the ramified primes of $K|\Q$, this shows that if $p$ is a ramified prime in $K$, then $p \equiv 1 \bmod 3$.

Note that we have found the same number of cyclic cubic fields of discriminant less than $10^{16}$, namely $15\,852\,618$.

\section{$p$-adic regulators of cyclic cubic fields}

Our conjecture is based on several observations made from the results of the computation and statistics of random matrices.
In this section, for each prime $p$ we will construct sets $M_p^\mathrm{split}, M_p^\mathrm{ram}, M_p^\mathrm{inert} \subseteq \GL_3(\Qbar)$ with the following property:
for each cyclic cubic field $K$ in which $p$ is split, there exists $A \in M_p^\mathrm{split}$ such that $R_p(K) = \det(A)$, and similarly for the other two cases.

We first note that there is a lower bound on the $p$-adic valuation of the regulator, depending on whether the field is ramified or not at $p$.

\begin{lem}
For a field cyclic cubic field $K$ we have
\[ v_p(R_p(K)) \geq \begin{cases} 1, &\text{ if $p$ is ramified in $K$}, \\ 2, &\text{ otherwise.}\end{cases} \]
\end{lem}

\begin{proof}
For a prime ideal $\mathfrak p$ of $K$ lying above $p$ let $\nu_\mathfrak p$ be the number of $p$-power roots of unity in the completion $K_\mathfrak p$.
By \cite[Appendix, Lemma 5]{Coa} we know that
\[ v_p(R_p(K)) \geq v_p({d(K)}^{\frac 1 2}) - v_p(3) - 1 + \sum_{\mathfrak p \mid p\mathcal O_K} ( f(\mathfrak p|p) + v_p(\nu_\mathfrak p)) \geq v_p({d(K)}^{\frac 1 2}) - 1 + \sum_{\mathfrak p \mid p\mathcal O_K} f(\mathfrak p|p), \]
where $f(\mathfrak p|p)$ is the inertia degree of $\mathfrak p$ over $p$.
As $v_p(d(K))=2$ if $p$ is ramified and $v_p(d(K))=0$ otherwise, the claim follows.
\end{proof}

To analyse the $p$-adic regulators, we will make use of the Galois module structure of the unit group:
we denote by $U_K$ the unit group of $\mathcal O_K$ and for a subgroup $U \subseteq U_K$ we set $U^\ast = U/\{-1, 1\}$.
Let $G = \langle \sigma \rangle$ be the Galois group of $K$ over $\Q$.
By \cite{Marko1996} we know that for every prime $p \neq 3$ there exists a unit $\varepsilon_p \in U_K$ such that $[ U_K^\ast : \langle G \varepsilon_p \rangle^\ast ]$ is not divisible by $p$.
It follows that $\langle \varepsilon_p,\sigma(\varepsilon_p) \rangle$ is a subgroup of $U_K$ of full rank and $[ U_K^\ast : \langle \varepsilon_p, \sigma(\varepsilon_p) \rangle^\ast]$ is prime to $p$.
In particular the $p$-adic regulator can be computed using $\varepsilon_p$ and $\sigma(\varepsilon_p)$.

\subsection{The split case}

For a prime $p >3$ let us define 
\[ M_p^{\mathrm{split}} = \left\{ \begin{pmatrix} 1 & 1 & 1 \\ a & b & -b -a \\ -b - a & a & b \end{pmatrix} \, \Bigg| \, a,b \in p\Z_p \right\} \subseteq \operatorname{Mat}_{3 \times 3}(\Z_p). \]

\begin{prop}
Let $K$ be a cyclic cubic field and $p > 3$ a prime which splits in $K$.
Then there exists $A \in M_p^{\mathrm{split}}$ such that $R_p(K) = \det(A)$.
\end{prop}

\begin{proof}
We set $\varepsilon = \varepsilon_p$.
Let $\mu \in \Z[X]$ be the minimal polynomial of $\varepsilon$.
Since $p$ splits we know that $\mu$ factors over $\Z_p$ as $\mu = (X-\tau_1(\varepsilon))(X-\tau_2(\varepsilon))(X-\tau_3(\varepsilon))$, where $\tau_1,\tau_2,\tau_3 \colon K \to \Qbar$ are the distinct $p$-adic embeddings.
As $f$ is also the minimal polynomial of $\sigma(\varepsilon)$, we may assume that $\tau_1(\varepsilon) = \tau_2(\sigma(\varepsilon))$, $\tau_2(\varepsilon) = \tau_3(\sigma(\varepsilon))$ and $\tau_3(\varepsilon) = \tau_1(\sigma(\varepsilon))$.
In particular, the $p$-adic regulator can be computed as
\[ R_p(K) = \det\begin{pmatrix} 1 & 1 & 1 \\ \log_p(\tau_1(\varepsilon)) & \log_p(\tau_2(\varepsilon)) & \log_p(\tau_3(\varepsilon)) \\ \log_p(\tau_3(\varepsilon)) & \log_p(\tau_1(\varepsilon)) & \log_p(\tau_2(\varepsilon))
\end{pmatrix}. \]
Now as $\tau_i(\varepsilon) \in 1 + p\Z_p$ we have $\log_p(\tau_i(\varepsilon)) \in p\Z_p$
(for all $n \geq 1$ the restriction of $\log_p$ induces an isomorphism $1 + p^n \Z_p \to p^n \Z_p$).
Finally $\varepsilon$ being a unit implies $\sum_{i=1}^3 \log_p(\tau_i(\varepsilon)) = 0$.
\end{proof}

The set $M_p^\mathrm{split}$ is isomorphic to $\Z_p^2$ via
\[ \Z_p^2 \longrightarrow M_p^\mathrm{split}, \; (a,b) \longmapsto \begin{pmatrix}  1 & 1 & 1 \\ pa & pb & -pb -pa \\ -pb - pa & pa & pb \end{pmatrix}. \]
Using this isomorphism, we equip $M_p^\mathrm{split}$ with its unique Haar measure $\mu$ such that $\mu(M_p^\mathrm{split}) = 1$ and consider the random variable $M_p^\mathrm{split} \to \R_{\geq 0}, \; A \mapsto v_p(\det(A))$.

\subsection{The ramified case}

Let $p$ be a prime with $p \equiv 1 \bmod 3$. 
As the group $\F_p^\times / (\F_p^\times)^3$ has order $3$ and $\Q_p$ contains third roots of unity, there are exactly three cubic ramified extensions $K_1,K_2,K_3$ of $\Q_p$, which are cyclic and totally ramified.
We denote their valuations rings by $\mathcal O_1,\mathcal O_2$ and $\mathcal O_3$. Let $\operatorname{Gal}(K_i|\Q_p) = \langle \sigma_i \rangle$.
For $i = 1,2,3$ we set
\[ M_{p,i}^\mathrm{ram} = \left\{ \begin{pmatrix}  1 & 1 & 1 \\ a & \sigma_i(a) & \sigma_i^2(a) \\ \sigma_i(a) & \sigma_i^2(a) & a \end{pmatrix} \; \Bigg| \; a \in \mathcal O_i, \; a \neq 0, \;\operatorname{Tr}_{K_i|\Q_p}(a) = 0 \, \right\} \subseteq \GL_3(\Qbar) \]
and $M_p^\mathrm{ram} = M_{p,1}^\mathrm{ram} \cup M_{p,2}^\mathrm{ram} \cup M_{p,3}^\mathrm{ram} \subseteq \GL_3(\Qbar)$.
Since for $a \in K_i$, $a \neq 0$, we have $\operatorname{Tr}_{K_i|\Q_p}(a) = 0$ if and only if $a \not\in \Q_p$ this is in fact a disjoint union $M_p^\mathrm{ram} = M_{p,1}^\mathrm{ram} \mathbin{\dot\cup} M_{p,2}^\mathrm{ram} \mathbin{\dot\cup} M_{p,3}^\mathrm{ram}$.

\begin{prop}\label{prop:inert}
Let $K$ be a cyclic cubic field and $p > 3$ a prime which is ramified in $K$.
Then there exists $A \in M_p^\mathrm{ram}$ such that $R_p(K) = \det(A)$.
\end{prop}

\begin{proof}
We set $\varepsilon = \varepsilon_p$.
Let $\mathfrak p$ be the prime ideal of $K$ lying above $p$ and $K_\mathfrak p$ be the $\mathfrak p$-adic completion.
As $\varepsilon \in K \subseteq K_\mathfrak p$ and $K_\mathfrak p|\Q_p$ is cyclic with Galois group $\langle \tau \rangle$, the $p$-adic embeddings of $\varepsilon$ are just $\varepsilon,\tau(\varepsilon), \tau^2(\varepsilon)$.
Thus
\[ R_p(K) = \begin{pmatrix} 1 & 1 & 1 \\ \log_p(\varepsilon) & \log_p(\tau(\varepsilon)) & \log_p(\tau^2(\varepsilon)) \\ \log_p(\tau(\varepsilon)) & \log_p(\tau^2(\varepsilon)) & \log_p(\varepsilon) \end{pmatrix}. \]
As $\tau$ commutes with $\log_p$ and $\varepsilon$ is a unit, the claim follows using the fact that $K_\mathfrak p \cong K_i$ for some $i = 1,2,3$.
\end{proof}

\subsection{The inert case}

For a prime $p > 3$ denote by $\smash{\Q_p^{(3)}}$ the unique unramified cubic extension of $\Q_p$ and by $\smash{\Z_p^{(3)}}$ its ring of integers.
This is a cyclic Galois extension with Galois group $\smash{\gal(\Q_p^{(3)}|\Q)} = \langle \tau \rangle$.
We define
\[ M_p^{\mathrm{inert}} = \left\{ \begin{pmatrix} 1 & 1 & 1 \\ a & \tau(a) & \tau^2(a) \\ \tau(a) & \tau^2(a) & a \end{pmatrix} \; \Bigg| \; a \in \Z_p^{(3)}, \, a \neq 0, \, \operatorname{Tr}_{\Q_p^{(3)}|\Q_p}(a) = 0 \right\}. \]

\begin{prop}
Let $K$ be a cyclic cubic field and $p > 3$ a prime which is inert in $K$.
Then there exists $A \in M_p^\mathrm{inert}$ such that $\det(A) = R_p(K)$.
\end{prop}

\begin{proof}
Let $\mathfrak p$ be the prime ideal of $K$ lying above $p$. Then $K_\mathfrak p \cong \smash{\Q_p^{(3)}}$.
Morevover, the Galois group $\gal(K_\mathfrak p|\Q_p)$ is cyclic and the restriction $\gal(K_\mathfrak p|\Q_p) \to \gal(K|\Q)$ is an isomorphism.
The proof is now the same as in Proposition~\ref{prop:inert}.
\end{proof}

Let $p > 3$ be a prime.
Note that for each $x$, the set $M_p^x$ is a subset of $\operatorname{Mat}_{3 \times 3}(\mathcal O)$ for some ring of integers $\mathcal O$ of a finite extension of $\Q_p$ and therefore is naturally equipped with canonical finite measure allowing us to speak of random variables on $M_p^x$.
Based on numerical observations, we make the following conjecture:

\begin{conjecture}\label{conj1}
Let $p > 3$ be a prime. For $x \in \{\mathrm{split},\,\mathrm{ram},\,\mathrm{inert}\}$ and $i \in \Z_{\geq 0}$ we have
\[ \lim_{D \to \infty} \frac{\#\{ K \in \mathcal K_p^x \mid v_p(R_p(K)) = i \text{ and } \lvert d(K) \rvert \leq D \}}{\#\{ K \in \mathcal K_p^x \mid \lvert d(K) \rvert \leq D \}} = \pr(X = i), \]
where $X \colon M_p^x \to \R_{\geq 0},\; A \mapsto v_p(\det(A))$.
\end{conjecture}

\section{Distribution of determinants of random matrices}

For each $x \in \{\mathrm{split}, \mathrm{ram}, \mathrm{inert}\}$ we now compute the distribution of the random variable $M_p^x \to \R_{\geq 0},\; A \mapsto v_p(\det(A))$ and use this to make Conjecture~\ref{conj1} about the distributions of the valuations of $p$-adic regulators more explicit and accessible to numerical investigations.

For a quadratic form $g \in \Z_p[X,Y]$ we denote by $X_g$ the random variable $X_g \colon \Z_p^2 \to \R_{\geq 0}, \, (a,b) \mapsto v_p(g(a,b))$. From the definition of equivalence of quadratic forms we immediately obtain the following consequence for the associated random variables.
If $h \in \Z_p[X,Y]$ is another quadratic form, we write $g \sim h$ if $g$ and $h$ are equivalent, that is, if there exists $B \in \operatorname{GL}_2(\Z_p)$ such that $g(X,Y) = h((X,Y)B)$.

\begin{lem}\label{lem:dist}
Let $g, h \in \Z_p[X,Y]$ be quadratic forms and $\alpha \in \Z_p$ such that $g \sim \alpha h$.
Then $X_g$ is in distribution equal to $X_h + v_p(\alpha)$.
\end{lem}

\begin{lem}
The random variable $M_p^\mathrm{split} \to \R_{\geq 0}, \; A \mapsto v_p(\det(A))$ is in distribution equal to $X_f + 2$, where $f = X^2 + 3Y^2$.
\end{lem}

\begin{proof}
For an element
\[ A = \begin{pmatrix}  1 & 1 & 1 \\ pa & pb & -pb -pa \\ -pb - pa & pa & pb \end{pmatrix} \]
of $M_p^\mathrm{split}$
we have $\det(A) = 3p^2 (a^2 + b^2 + ab)$.
Therefore $M_p^\mathrm{split} \to \R_{\geq 0}, \; A \mapsto v_p(\det(A))$ is in distribution equal to $X_h$, where $h = 3p^2(X^2 + Y^2 + XY) \in \Z_p[X,Y]$.
The claim follows from applying Lemma~\ref{lem:dist}.
\end{proof}

\begin{lem}
Let $p \equiv 1 \bmod 3$. The random variable $M_p^\mathrm{ram} \to \R_{\geq 0}, \; A \mapsto v_p(\det(A))$ is in distribution equal to $X_f + 1$,
where $f = X^2 + 3Y^2$.
\end{lem}

\begin{proof}
Since $M_p^\mathrm{ram}$ is the disjoint union of the $M_{p,i}^\mathrm{ram}$ it is sufficient to show it for each random variable $M_{p,i}^\mathrm{ram} \to \R_{\geq 0}, \, A \mapsto v_p(\det(A))$. Let $K = K_i$, $\mathcal O = \mathcal O_i$, $\sigma = \sigma_i$ and denote by $\zeta \in \mathcal O$ a third root of unity.
As $K|\Q_p$ is totally ramified we can find a primitive element $\alpha \in \mathcal O$ such that $\mathcal O = \Z_p[\alpha]$ and $\alpha$ has minimal polynomial $X^3 - c$ with $c \in \Z_p$ and $v_p(c) = 1$.
In particular we have a bijection
\[ \varphi \colon \Z_p^2 \longrightarrow M_{p,i}^\mathrm{ram} \cup \{ 0 \}, \; (a,b) \longmapsto \begin{pmatrix}  1 & 1 & 1 \\ \beta & \sigma(\beta) & \sigma^2(\beta) \\ \sigma(\beta) & \sigma^2(\beta) & \beta \end{pmatrix}, \; \beta = a\alpha + b \alpha^2. \]
We may assume that $\sigma(\alpha) = \zeta\alpha$. 
Let $A = \varphi(a,b)$.
An easy calculations shows that $\det(A) = \det(\varphi(a,b)) = -9abc$. As $p \nmid 9$ and $v_p(c) = 1$ we conclude that $M_{p,i}^\mathrm{ram} \to \R_{\geq 0},\; A \mapsto v_p(\det(A))$ is in distribution equal to $X_g + 1$, where $g = XY$.
As $p \equiv 1 \bmod 3$ we know that $-3$ is a quadratic residue modulo $p$, that is, there exists $s \in \Z$ such that $s^2 \equiv -3 \bmod p$. Consider the matrix
\[ B = \begin{pmatrix} 1 & 1 \\ s & -s \end{pmatrix}. \]
Now $\det(B) = -2s \in \Z_p^\times$ and therefore $g = XY \sim (X+sY)(X-sY) = X^2 - s^2 Y^2 = X^2 + 3Y^2$.
\end{proof}

We now turn to the inert case.
Unfortunately, in this case we only have the following conjecture, which basically states that the matrices used to compute the $p$-adic regulator, are uniformly distributed.

\begin{conjecture}\label{conj:inert}
The random variable $M_p^\mathrm{inert} \to \R_{\geq 0}, \; A \mapsto v_p(\det(A))$ is in distribution equal to $X_f + 2$, where $f = X^2 + 3Y^2$.
\end{conjecture}

\begin{rem}
While we could not prove this conjecture, we have numerically tested it for a large number of primes $p$.
Moreover in Section~5 we have collected numerical evidence for Conjecture~\ref{conj}, which is a consequence of Conjecture~\ref{conj:inert}, therefore also justifying Conjecture~\ref{conj:inert}.
\end{rem}

Let $M = M_p^{x}$ for some $x \in \{\mathrm{split}, \mathrm{ram}, \mathrm{inert}\}$.
We now want to compute the probability that the valuation of the determinant of a randomly chosen matrix in $M$ has given value.
Since we have seen that $M \to \R_{\geq 0}, \; A \mapsto v_p(\det(A))$ is in distribution equal to $X_f + 1$ or $X_f + 2$ respectively, where $f = X^2 + 3Y^2$, it is sufficient to determine the probability $\pr(X_f = i)$, $i \in \Z_{\geq 0}$.
Now $\pr(X_f = i) = \mu(M_i)$, where
\[ M_i = \{ (a,b) \in \Z_p \times \Z_p \mid v_p(f(a,b)) = i \} \]
and $\mu$ is the Haar measure on $\Z_p \times \Z_p$ with $\mu(\Z_p \times \Z_p) = 1$. Note that since our form $f$ has integer coefficients, for any commutative ring we have a unique evaluation map $R \times R \to R$, $(a,b) \mapsto f(a,b)$ compatible with the unique morphism $\Z \to R$.

\begin{lem}\label{lem:2}
For $i \geq 0$ we have
\[ \pr(X_f = i ) = \lim_{k\to \infty} \frac{1}{p^{2k}} \#\{(\bar x, \bar y) \in (\Z/p^k\Z)^2 \mid v_p(f(\bar x,\bar y)) = i \}. \]
\end{lem}

\begin{proof}
For $k \geq 1$ let us denote by $\pi_k \colon \Z_p^2 \to (\Z/p^k\Z)^2$ the canonical projection and by $\mu_k$ the normalised counting measure on $(\Z/p^k\Z)^2$.
As $(\Z_p)^2$ is the projective limit of the measure spaces $((\Z/p^k\Z)^2)_{k \geq 1}$, we have
\[ \mu(M_i) = \lim_{k \to \infty} \mu_k(\pi_k(M_i)). \]
Now $\pi_k(M_i) = \{(\bar x, \bar y) \in (\Z/p^k\Z)^2 \mid v_p(f(x,y)) = i \}$ and for $k$ large enough we have $v_p(f(\bar x, \bar y)) = i$ if and only if $v_p(f(x,y)) = i$.
\end{proof}

\begin{lem}
For $k \geq i$ we have
\[ \# \{ (x,y) \in (\Z/p^k \Z)^2 \mid x^2 + 3y^2 \equiv 0 \bmod p^i \} \\ = p^{2(k-i)}  \# \{ (x,y) \in (\Z/p^i \Z)^2 \mid x^2 + 3y^2 = 0 \}. \]
\end{lem}

For a commutative ring $R$ let us denote by $X(R)$ the set $\{ (x,y) \in R^2 \mid x^2 + 3y^2 = 0 \}$. Our aim is to compute the right hand side in Lemma~\ref{lem:2} by calculating $\#X(R)$ for the residue rings $R = \Z/p^k\Z$. By the properties of $f$, this investigation naturally splits into two cases.

\subsection{The case $p \equiv 1 \bmod 3$}

\begin{lem}\label{lem:3}
Let $i \geq 0$ and $p \equiv 1 \bmod 3$. Then the following hold:
\begin{enumerate}[(i)]
\item
  We have $\#X(\Z/p^i\Z) = p^{i-1}((p-1)i + p)$.
\item
  We have 
  \[ \mu(M_i) = \frac{i+1}{p^i}\left( 1 - \frac 1 p \right)^2.\]
\end{enumerate}
\end{lem}

\begin{proof}
(i): As $p$ is congruent $1$ modulo $3$, the Legendre symbol $(\frac{-3} p )$ evaluates to $1$, that is, $-3$ is a quadratic residue modulo $p$.
The prime $p$ being odd, we can find therefore $\alpha \in \Z/p^i \Z$ such that $\alpha^2 = -3$ in $\Z/p^i \Z$.
In particular our binary quadratic form $f$ factors as $X^2 + 3Y^2 = (X+\alpha Y)(X-\alpha Y)$.
As
\begin{align*} \{ (x,y) \in (\Z/p^i\Z)^2 \mid x^2 + 3y^2 = 0 \} & \longrightarrow \{ (u,v) \in (\Z/p^i\Z)^2 \mid u \cdot v = 0\},\\  (x,y) &\longmapsto (x+\alpha y, x - \alpha y), \end{align*}
is a bijection, it is sufficient to count the cardinality of the set on the right hand side.
We have
\[ \{ (u,v) \in (\Z/p^i\Z)^2 \mid u \cdot v = 0\} =
U \mathbin{\dot\cup} \mathop{\dot\bigcup}\limits_{l=1}^{i-1} (E_l \times E_{\geq i -l}) , \]
where $E_l$ resp. $E_{\geq i - l}$ denotes the set of elements with valuation equal to $l$ resp. greater or equal to $i-l$ and $U$ is the set $\{ (u,v) \in (\Z/p^i\Z)^2 \mid u = 0 \text{ or } v = 0\}$.
As $\# E_{k} = (p-1)p^{i-k-1}$ we obtain
\begin{align*}
\#\{ (u,v) \in (\Z/p^i\Z)^2 \mid u \cdot v = 0\} & = \# U + \sum_{l=1}^{i-1} \# E_l \sum_{j=i-l}^{i-1} \# E_j \\
& = 2p^i - 1 + \sum_{l=1}^{i-1} (p-1) p^{i-l-1} \sum_{j=i-l}^{i-1} (p-1)p^{i-j-1}.
\end{align*}
Now
\[ \sum_{j=i-l}^{i-1} p^{i-j-1} = \sum_{j=0}^{l-1} p^{l-j-1} = p^{l-1} \sum_{j=0}^{l-1} \left(\frac 1 p \right)^j = \frac{p^l-1}{p-1}, \]
and
\[ (p-1)\sum_{l=1}^{i-1} p^{i-l-1} (p^l - 1) = (p-1)p^{i-1}\left(i - \sum_{l=0}^{i-1} \left( \frac 1 p \right)^l \right) = (p-1)p^{i-1} i - p^i + 1.\]
Thus
\[ \#X(\Z/p^i\Z) = 2p^i - 1 + (p-1)p^{i-1}i - p^i + 1 = p^{i-1}((p-1)i + p). \]
(ii): First note that for $k \geq i$ we have
\begin{align*} & \phantom{ = }\,\,\, \{ (\bar x, \bar y) \in (\Z/p^k \Z)^2 \mid v_p(f(\bar x , \bar y)) = i\} \\ & =  \{ (\bar x, \bar y) \in (\Z/p^k \Z)^2 \mid f(\bar x, \bar y) \equiv 0 \bmod p^i \text{ and } f(\bar x, \bar y) \not\equiv 0 \bmod p^{i+1} \} \\ 
& = \{ (\bar x, \bar y ) \in (\Z/p^k \Z)^2 \mid f(\bar x ,\bar y) \equiv 0 \bmod p^i \} \setminus \{ (\bar x, \bar y ) \in (\Z/p^k \Z)^2 \mid f(\bar x ,\bar y) \not\equiv 0 \bmod p^{i+1} \} \\
& = p^{2(k-i)} \cdot \# X(\Z/p^i \Z) - p^{2(k-i+1)} \cdot \# X(\Z/p^{i+1}\Z).
\end{align*}
Thus we conclude that $\mu(M_i)$ is equal to
\[ \lim_{k \to \infty} \frac{1}{p^{2k}} \left( p^{2k-i-1}((p-1)i+p) - p^{2k -(i+1) -1}((p-i)(i+1)+p)\right) = \frac{i+1}{p^i}\left(1- \frac 1 p \right)^2. \]
\end{proof}

\subsection{The case $p \equiv 2 \bmod 3$}

\begin{lem}
Let $i \geq 0$ and $p \equiv 2 \bmod 3$.
Then the following hold:
\begin{enumerate}[(i)]
\item
We have $\# X(\Z/p^i\Z) = p^{2i - 2\lceil \frac i 2 \rceil}$.
\item
We have
\[ \mu(M_i) = \begin{cases} 0, &\text{ if $i$ is odd},\\
\frac 1 {p^i} \left( 1- \frac  1 {p^2} \right), &\text{ if $i$ is even.}\end{cases} \]
\end{enumerate}
\end{lem}

\begin{proof}
First note that by assumption, $-3$ is not a square modulo $p^i$ for all $i \geq 1$.
(i): Assume that $(x,y) \in X(\Z/p^i \Z)$. If $y$ is a unit, then $-3 = (x/y)^2$ is a square modulo $p^i$, a contradiction.
Now assume $2 v_p(y) < j$, that is, $j < \lceil \frac i 2 \rceil$.
We write $x = \varepsilon \cdot p^j$, $y = \varepsilon' \cdot p^{j'}$ with units $\varepsilon, \varepsilon'$.
Denoting by $s$ the minimum $\min(2j, 2j')$ we obtain $i > s$ and moreover
\[ \varepsilon^2 p^{2j - s} \equiv -3 \varepsilon'^2 p^{2j' - s} \bmod p^{i-s}. \]
Since $s = 2j$ or $s = 2j'$ this implies that $-3$ is a square modulo $p^{i-s}$, a contradiction.
This shows that all elements $(x,y) \in X(\Z/p^i\Z)$ have to satisfy $y^2 = 0$, $x^2 = 0$ and we necessarily have $X(\Z/p^i\Z) = \# \{ x \in \Z/p^i \Z \mid x^2 = 0 \}^2 $.
As \[ \{ x \in \Z/p^i \Z \mid x^2 = 0 \} =  \{ 0 \} \mathbin{\dot\cup} \mathop{\dot\bigcup}\limits_{l = \lceil \frac i 2 \rceil}^{i-1} E_l  \]
we obtain
\[ \#\{ x \in \Z/p^i \Z \mid x^2 = 0 \} = 1 + \sum_{l = \lceil \frac i 2 \rceil}^{i-1} \# E_l = 1 + p^{i - \lceil \frac i 2 \rceil} - 1 = p^{i - \lceil \frac i 2 \rceil}. \]
(ii): As in the proof of Lemma~\ref{lem:3} we obtain
\[ \mu(M_i) = \lim_{k \to \infty} \frac{1}{p^{2k}}\left( p^{2k -2i} p^{2i - 2\lceil \frac i 2 \rceil } - p^{2k - 2i - 2} p^{2i + 2 - 2\lceil \frac{i+1} 2 \rceil}\right) = p^{-2 \lceil \frac i 2 \rceil} - p^{-2 \lceil \frac{i+1} 2 \rceil}. \]
\end{proof}

We can now combine these results:

\begin{thm}
  Assume that Conjectures~\ref{conj1} and~\ref{conj:inert} hold. Then for all primes $p > 3$ the following hold:
  \begin{enumerate}[(i)]
  \item
  If $p \equiv 2 \bmod 3$, then $v_p(R_p(K)) \in 2 \Z$ for all $K \in \mathcal K$ and for $i \geq 0$ we have
  \[ \lim_{D \to \infty} \frac{\#\{ K \in \mathcal K(D) \mid v_p(R_p(K)) = 2i \}}{\# \mathcal K(D)} = \frac 1 {p^{2i - 2}} \left( 1- \frac  1 {p^2} \right). \]
  \item
  If $p \equiv 1 \bmod 3$, then for $i \geq 0$ we have
  \[ \lim_{D \to \infty} \frac{\#\{ K \in \mathcal K_p^{\mathrm{un}}(D) \mid v_p(R_p(K)) = i \}}{\#\mathcal K_p^{\mathrm{un}}(D)} = \frac{i-1}{p^{i-2}}\left( 1 - \frac 1 p \right)^2. \]
  and
  \[ \lim_{D \to \infty} \frac{\#\{ K \in \mathcal K_p^{\mathrm{ram}}(D) \mid v_p(R_p(K)) = i\}}{\# \mathcal K_p^{\mathrm{ram}}(D)} =  \frac{i}{p^{i-1}}\left( 1 - \frac 1 p \right)^2.  \]
  \end{enumerate}
\end{thm}

\section{Results and heuristics}

We compute the $p$-adic valuation of $p$-adic regulators for all cyclic cubic extensions of $\Q$ with discriminant up to $10^{16}$ for $p<100$. The distribution table of $p$-adic valuation of $p$-adic regulators for four values of $p$ ($5$, $7$, $11$ and $13$) are provided here and compared with Conjecture~\ref{conj}, see Tables~\ref{tab:5}, ~\ref{tab:7},~\ref{tab:11} and~\ref{tab:13}. As $p$ increases, both the expected and observed number of fields with higher valuations drop significantly, reducing their effectiveness to provide heuristics for our conjecture. The distribution tables for these remaining values of $p$ are available in the appendix.

\subsection{$p = 5$ and $p= 11$}
The results are in Table~\ref{tab:5} and Table~\ref{tab:11} respectively. As $5 \equiv 2 \bmod 3$ and $11 \equiv 2 \bmod 3$, all fields are unramified. The distribution match closely with case 1 of Conjecture~\ref{conj}.

\begin{table}[!ht]
\small
\centering
 \caption{Distribution of $5$-adic valuation of $5$-adic regulators}
 \label{tab:5}
 \begin{tabular}{lrrrrrrrrrr}
    $D$ & $\#\mathcal K_5(D)$ & $2$ & $3$ & $4$ & $5$ & $6$ & $7$ & $8$ & $9$ & $10$ \\
    \hline
    \rule{0pt}{10pt}
    $10^8$      & 1\,592
                & .95979
                & 0
                & .03831
                & 0
                & .00125
                & 0
                & 0
                & 0
                & .628E-3
                \\
    $10^{10}$   & 15\,851
                & .96221
                & 0
                & .03646
                & 0
                & .00119
                & 0
                & .630E-4
                & 0
                & .630E-4
                \\
    $10^{12}$   & 158\,542
                & .96146
                & 0
                & .03691
                & 0
                & .00153
                & 0
                & .693E-4
                & 0
                & .126E-4
                \\
    $10^{14}$   & 1\,585\,249
                & .96026
                & 0
                & .03816
                & 0
                & .00151
                & 0
                & .599E-4
                & 0
                & .126E-5
                \\
    $10^{16}$   & 15\,852\,618
                & .96006
                & 0
                & .03834
                & 0
                & .00152
                & 0
                & .606E-4
                & 0
                & .164E-5
                \\
    \hline
   \multicolumn{2}{c}{Conjecture~\ref{conj}}
                & .96000
                & 0
                & .03840
                & 0
                & .00153
                & 0
                & .614E-4
                & 0
                & .245E-5
     \end{tabular}
\end{table}

\begin{table}[ht]
\small
\centering
 \caption{Distribution of $11$-adic valuation of $11$-adic regulators}
 \label{tab:11}
 \begin{tabular}{lrrrrrrrrr}
    $D$ & $\#\mathcal K_{11}(D)$ & $2$ & $3$ & $4$ & $5$ & $6$ & $7$ & $8$ \\
    \hline
    \rule{0pt}{10pt}
    $10^8$      & 1\,592
                & .989949
                & 0
                & .010050
                & 0
                & 0
                & 0
                & 0
                \\
    $10^{10}$   & 15\,851
                & .992871
                & 0
                & .007128
                & 0
                & 0
                & 0
                & 0
                \\
    $10^{12}$   & 158\,542
                & .992361
                & 0
                & .007594
                & 0
                & .441E-4
                & 0
                & 0
                \\
    $10^{14}$   & 1\,585\,249
                & .991884
                & 0
                & .008045
                & 0
                & .693E-4
                & 0
                & .630E-6
                \\
    $10^{16}$   & 15\,852\,618
                & .991780
                & 0
                & .008154
                & 0
                & .643E-4
                & 0
                & .504E-6
                \\
    \hline
   \multicolumn{2}{c}{Conjecture~\ref{conj}}
                & .991735
                & 0
                & .008196
                & 0
                & .677E-4
                & 0
                & .559E-6
    \end{tabular}
\end{table}

\subsection{$p = 7$ and $p = 13$}
The results are in Table~\ref{tab:7} and Table~\ref{tab:13} respectively. Since $7 \equiv 1 \bmod 3$ and $13 \equiv 1 \bmod 3$, some of the fields in question are ramified, so the overall distribution is difficult to predict. However, once we separate the ramified and unramified extensions, we see again that the distributions match Conjecture~\ref{conj} for both parts.

\begin{table}[ht]
\small
\centering
\caption{Distribution of $7$-adic valuation of $7$-adic regulators}
\label{tab:7}
 \begin{tabular}{lrrrrrrrrr}
    $D$ & $\#\mathcal K_7^{\mathrm{ram}}(D) $ & $1$ & $2$  & $3$ & $4$ & $5$ & $6$ & $7$ & $8$ \\
    \hline
    \rule{0pt}{10pt}
    $10^8$      & 357
                & .7591
                & .1652
                & .0560
                & .0140
                & .0056
                & 0
                & 0
                & 0
                \\
    $10^{10}$   & 3\,513
                & .7514
                & .2012
                & .0384
                & .0071
                & .0014
                & .28E-3
                & 0
                & 0 
                \\
    $10^{12}$   & 35\,237
                & .7418
                & .2060
                & .0406
                & .0089
                & .0020
                & .31E-3
                & .28E-4
                & .28E-4
                \\
    $10^{14}$   & 352\,231
                & .7354
                & .2094
                & .0440
                & .0086
                & .0015
                & .30E-3
                & .45E-4
                & .11E-4
                \\
    $10^{16}$   & 3\,522\,753
                & .7352
                & .2095
                & .0447
                & .0085
                & .0015
                & .28E-3
                & .42E-4
                & .06E-4
                \\
    \hline
   \multicolumn{2}{c}{Conjecture~\ref{conj}}
                & .7346
                & .2099
                & .0449
                & .0085
                & .0015
                & .26E-3
                & .43E-4
                & .07E-4
     \end{tabular}

 \begin{tabular}{lrrrrrrrrr}
    $D$ & $\#\mathcal K_7^\mathrm{un}(D)$ & $2$  & $3$ & $4$ & $5$ & $6$ & $7$ & $8$ & $9$ \\
    \hline
    \rule{0pt}{10pt}
    $10^8$      & 1\,235
                & .7165
                & .2275
                & .0485
                & .0056
                & .0016
                & 0
                & 0
                & 0
                \\
    $10^{10}$   & 12\,338
                & .7323
                & .2153
                & .0428
                & .0076
                & .0015
                & .16E-3
                & 0
                & 0 
                \\
    $10^{12}$   & 123\,305
                & .7363
                & .2093
                & .0439
                & .0082
                & .0018
                & .26E-3
                & .40E-4
                & .16E-4
                \\
    $10^{14}$   & 1\,233\,018
                & .7349
                & .2098
                & .0448
                & .0084
                & .0015
                & .28E-3
                & .42E-4
                & .08E-4
                \\
    $10^{16}$   & 12\,329\,865
                & .7347
                & .2099
                & .0449
                & .0085
                & .0015
                & .26E-3
                & .43E-4
                & .08E-4
                \\
    \hline
   \multicolumn{2}{c}{Conjecture~\ref{conj}}
                & .7346
                & .2099
                & .0449
                & .0085
                & .0015
                & .26E-3
                & .43E-4
                & .07E-4
     \end{tabular}
   \end{table}

\begin{table}[ht]
\small
\centering
\caption{Distribution of $13$-adic valuation of $13$-adic regulators}
 \begin{tabular}{lrrrrrrrr}
 \label{tab:13}
    $D$ & $\#\mathcal K_{13}^\mathrm{ram}(D)$ & $1$ & $2$ & $3$ & $4$ & $5$ & $6$ & $7$  \\
    \hline
    \rule{0pt}{10pt}
    $10^8$      & 221
                & .8778
                & .1085
                & .0045
                & .0045
                & .45E-2
                & 0
                & 0 
                \\
    $10^{10}$   & 2\,117
                & .8611
                & .1232
                & .0122
                & .0014
                & .14E-2
                & 0
                & 0
                \\
    $10^{12}$   & 21\,137
                & .8542
                & .1289
                & .0143
                & .0019
                & .37E-3
                & 0
                & 0
                \\
    $10^{14}$   & 211\,369
                & .8512
                & .1315
                & .0152
                & .0017
                & .22E-3
                & .09E-4
                & 0
                \\
    $10^{16}$   & 2\,113\,583
                & .8519
                & .1311
                & .0150
                & .0015
                & .15E-3
                & .12E-4
                & .14E-5
                \\
    \hline
   \multicolumn{2}{c}{Conjecture~\ref{conj}}
                & .8520
                & .1310
                & .0151
                & .0015
                & .14E-3
                & .13E-4
                & .12E-5
    \end{tabular}
    
    \vspace{1em}
    
     \begin{tabular}{lrrrrrrrrr}
    $D$ & $\#\mathcal K_{13}^\mathrm{un}(D)$ & $2$ & $3$ & $4$ & $5$ & $6$ & $7$ & $8$ & $9$ \\
    \hline
    \rule{0pt}{10pt}
    $10^8$      & 1\,371
                & .8460
                & .1356
                & .0145
                & .0036
                & 0
                & 0
                & 0
                & 0 
                \\
    $10^{10}$   & 13\,734
                & .8516
                & .1311
                & .0158
                & .0013
                & .07E-3
                & 0
                & 0 
                & 0
                \\
    $10^{12}$   & 137\,405
                & .8525
                & .1303
                & .0153
                & .0015
                & .18E-3
                & 0
                & 0
                & 0
                \\
    $10^{14}$   & 1\,373\,880
                & .8521
                & .1307
                & .0154
                & .0015
                & .15E-3
                & .14E-4
                & .14E-5
                & .72E-6
                \\
    $10^{16}$   & 13\,739\,035
                & .8520
                & .1310
                & .0151
                & .0015
                & .15E-3
                & .13E-4
                & .09E-5
                & .21E-6
                \\
    \hline
   \multicolumn{2}{c}{Conjecture~\ref{conj}}
                & .8520
                & .1310
                & .0151
                & .0015
                & .14E-3
                & .13E-4
                & .12E-5
                & .10E-6
    \end{tabular}
\end{table}

\section{Acknowledgements}
The authors would like to thank Claus Fieker for many helpful discussions and remarks. 

\clearpage

\section*{\refname}
\bibliographystyle{elsart-num-sort}
\bibliography{C3_bib}

\newpage
\appendix
\section{Tables for primes $17 \leq p \leq 100$}

\begin{table}[!ht]
\small
\centering
\caption{Distribution of $p$-adic valuation of $p$-adic regulators in case $p$ is unramified}
 \begin{tabular}{lrrrrrrrr}
    $p$ & $\#\mathcal K_{p}^\mathrm{un}(10^{16})$ & $2$ & $3$ & $4$ & $5$ & $6$ & $7$ & $8$   \\
    \hline
                $17$
                & 15\,852\,618
                & .99654
                & 0
                & .00344
                & 0
                & .113E-4
                & 0
                & .630E-7
                \\
    
    \multicolumn{2}{c}{Conjecture~\ref{conj}}
                & .99653
                & 0
                & .00344
                & 0
                & .119E-4
                & 0
                & .412E-7
                \\
    \hline
                $19$
                & 14\,342\,965
                & .89755
                & .09443
                & .00744
                & .524E-3
                & .358E-4
                & .202E-5
                & .139E-6
                \\
    \multicolumn{2}{c}{Conjecture~\ref{conj}}
                & .89750
                & .09447
                & .00745
                & .523E-3
                & .344E-4
                & .217E-5
                & .133E-6
                \\
    \hline
                $23$
                & 15\,852\,618
                & .99810
                & 0
                & .00188
                & 0
                & .422E-5
                & 0
                & 0
                \\
    \multicolumn{2}{c}{Conjecture~\ref{conj}}
                & .99810
                & 0
                & .00188
                & 0
                & .356E-5
                & 0
                & .674E-8
                \\
    \hline
                $29$
                & 15\,852\,618
                & .99881
                & 0
                & .00118
                & 0
                & .113E-5
                & 0
                & 0
                \\
    \multicolumn{2}{c}{Conjecture~\ref{conj}}
                & .99881
                & 0 
                & .00118
                & 0
                & .141E-5
                & 0
                & .167E-8
                \\
    \hline
                $31$
                & 14\,891\,921
                & .93654
                & .06038
                & .00293
                & .126E-3
                & .530E-5
                & .671E-7
                & 0
                \\
    \multicolumn{2}{c}{Conjecture~\ref{conj}}
                & .93652
                & .06042
                & .00292
                & .125E-3
                & .507E-5
                & .196E-6
                & .738E-8
                \\
    \hline
                $37$
                & 15\,039\,731
                & .94658
                & .05126
                & .00207
                & .742E-4
                & .265E-5
                & .132E-6
                & 0
                \\
    \multicolumn{2}{c}{Conjecture~\ref{conj}}
                & .94667
                & .05117
                & .00207
                & .747E-4
                & .252E-5
                & .819E-7
                & .258E-8
                \\
    \hline
                $41$
                & 15\,852\,618
                & .99940
                & 0
                & .597E-3
                & 0
                & .189E-6
                & 0
                & 0
                \\
    \multicolumn{2}{c}{Conjecture~\ref{conj}}
                & .99940
                & 0
                & .594E-3
                & 0
                & .353E-6
                & 0
                & .210E-9
                \\
    \hline
                $43$
                & 15\,148\,101
                & .95407
                & .04432
                & .00154
                & .477E-4
                & .184E-5
                & .660E-7
                & 0 
                \\
    \multicolumn{2}{c}{Conjecture~\ref{conj}}
                & .95402
                & .04437
                & .00154
                & .479E-4
                & .139E-5
                & .389E-7
                & .105E-8
                \\
    \hline
                $47$
                & 15\,852\,618
                & .99954
                & 0
                & .456E-3
                & 0
                & .126E-6
                & 0 
                & 0
                \\
    \multicolumn{2}{c}{Conjecture~\ref{conj}}
                & .99954
                & 0
                & .452E-3
                & 0 
                & .204E-6
                & 0
                & .927E-10
                \\
    \hline
                $53$
                & 15\,852\,618
                & .99965
                & 0
                & .349E-3
                & 0
                & .630E-7
                & 0
                & 0
                \\
    \multicolumn{2}{c}{Conjecture~\ref{conj}}
                & .99964
                & 0
                & .355E-3
                & 0
                & .126E-6
                & 0
                & .451E-10
                \\
    \hline
                $59$
                & 15\,852\,618
                & .99971
                & 0
                & .287E-3
                & 0
                & .630E-7
                & 0
                & 0
                \\
    \multicolumn{2}{c}{Conjecture~\ref{conj}}
                & .99971
                & 0
                & .287E-3
                & 0
                & .825E-7
                & 0
                & .237E-10
               \\
    \hline
                $61$
                & 15\,349\,425
                & .96752
                & .03167
                & .790E-3
                & .172E-4
                & .456E-6
                & 0
                & 0
                \\
    \multicolumn{2}{c}{Conjecture~\ref{conj}}
                & .96748
                & .03172
                & .780E-3
                & .170E-4
                & .349E-6
                & .687E-8
                & .131E-9
                \\
    \hline
                $67$
                & 15\,393\,169
                & .96752
                & .03167
                & .790E-3
                & .172E-4
                & .456E-6
                & 0
                & 0
                \\
    \multicolumn{2}{c}{Conjecture~\ref{conj}}
                & .97027
                & .02906
                & .648E-3
                & .131E-4
                & .129E-6
                & 0
                & 0
                \\
    \hline
                $71$
                & 15\,852\,618
                & .99980
                & 0
                & .194E-3
                & 0
                & .126E-6
                & 0
                & 0
                \\
    \multicolumn{2}{c}{Conjecture~\ref{conj}}
                & .99980
                & 0
                & .198E-3
                & 0
                & .393E-7
                & 0
                & .780E-11
                \\
    \hline
                $73$ 
                & 15\,429\,905
                & .97276
                & .02667
                & .544E-3
                & .952E-5
                & .129E-6
                & 0
                & 0 
                \\
    \multicolumn{2}{c}{Conjecture~\ref{conj}}
                & .97279
                & .02665
                & .547E-3
                & .100E-4
                & .171E-6
                & .281E-8
                & .449E-10
                \\
    \hline
                $79$
                & 15\,461\,291
                & .97487
                & .02465
                & .465E-3
                & .711E-4
                & .646E-7
                & 0
                & 0 
                \\
    \multicolumn{2}{c}{Conjecture~\ref{conj}}
                & .97484
                & .02467
                & .468E-3
                & .790E-5
                & .125E-6
                & .190E-8
                & .280E-10
                \\
    \hline
                $83$
                & 15\,852\,618
                & .99985
                & 0
                & .145E-3
                & 0
                & 0     
                & 0
                & 0
                \\
    \multicolumn{2}{c}{Conjecture~\ref{conj}}
                & .99985
                & 0
                & .145E-3
                & 0
                & .210E-7
                & 0
                & .305E-11
                \\
    \hline
                $89$
                & 15\,852\,618
                & .99986
                & 0
                & .130E-3
                & 0
                & 0     
                & 0
                & 0
                \\
    \multicolumn{2}{c}{Conjecture~\ref{conj}}
                & .99987
                & 0
                & .126E-3
                & 0
                & .159E-7
                & 0
                & .201E-11
                \\
    \hline
                $97$
                & 15\,532\,411
                & .97948
                & .02020
                & .313E-3
                & .405E-5
                & .128E-6
                & 0
                & 0
                \\
    \multicolumn{2}{c}{Conjecture~\ref{conj}}
                & .97948
                & .02019
                & .312E-3
                & .429E-5
                & .553E-7
                & .684E-9
                & .823E-11
    \end{tabular}
\end{table}

\begin{table}[!ht]
\small
\centering
\caption{Distribution of $p$-adic valuation of $p$-adic regulators in case $p$ is ramified}
 \begin{tabular}{lrrrrrrrr}
    $p$ & $\#\mathcal K_{p}^\mathrm{ram}(10^{16})$ & $1$ & $2$ & $3$ & $4$ & $5$ & $6$ & $7$  \\
    \hline
                $19$
                & 1\,509\,653
                & .89770
                & .09423
                & .00745
                & .555E-3
                & .384E-4
                & .397E-5
                & .662E-6
                \\
   \multicolumn{2}{c}{Conjecture~\ref{conj}}
                & .89750
                & .09447
                & .00745
                & .523E-3
                & .344E-4
                & .217E-5
                & .133E-6
                \\
    \hline
                $31$
                & 960\,697
                & .93663
                & .06027
                & .00295
                & .123E-3
                & .624E-5
                & .208E-5
                & 0
                \\
    \multicolumn{2}{c}{Conjecture~\ref{conj}}
                & .93652
                & .06042
                & .00292
                & .125E-3
                & .507E-5
                & .196E-6
                & .738E-8
                \\
    \hline
                $37$
                & 812\,887
                & .94678
                & .05099
                & .00213
                & .873E-4
                & .246E-5
                & 0
                & 0
                \\
    \multicolumn{2}{c}{Conjecture~\ref{conj}}
                & .94667
                & .05117
                & .00207
                & .747E-4
                & .252E-5
                & .819E-7
                & .258E-8
                \\
    \hline
                $43$
                & 704\,517
                & .95354
                & .04489
                & .00151
                & .525E-4
                & 0
                & 0
                & 0
                \\
    \multicolumn{2}{c}{Conjecture~\ref{conj}}
                & .95402
                & .04437
                & .00154
                & .479E-4
                & .139E-5
                & .389E-7
                & .105E-8
                \\   
    \hline
                $61$
                & 503\,193
                & .96740
                & .03176
                & .810E-3
                & .218E-4
                & .198E-5
                & 0
                & 0
                \\
    \multicolumn{2}{c}{Conjecture~\ref{conj}}
                & .96748
                & .03172
                & .780E-3
                & .170E-4
                & .349E-6
                & .687E-8
                & .131E-9
                \\
    \hline
                $67$
                & 459\,449
                & .97061
                & .02875
                & .607E-3
                & .152E-4
                & 0
                & 0
                & 0
                \\
    \multicolumn{2}{c}{Conjecture~\ref{conj}}
                & .97037
                & .02896
                & .648E-3
                & .129E-4
                & .240E-6
                & .431E-8
                & .750E-10
                \\
    \hline
                $73$
                & 422\,713
                & .97262
                & .02679
                & .570E-3
                & .118E-4
                & .236E-5
                & 0
                & 0
                \\
    \multicolumn{2}{c}{Conjecture~\ref{conj}}
                & .97279
                & .02665
                & .547E-3
                & .100E-4
                & .171E-6
                & .281E-8
                & .449E-10
                \\
    \hline
                $79$ 
                & 391\,327
                & .97438
                & .02512
                & .475E-3
                & .102E-4
                & 0
                & 0 
                & 0
                \\
    \multicolumn{2}{c}{Conjecture~\ref{conj}}
                & .97484
                & .02467
                & .468E-3
                & .790E-5
                & .125E-6
                & .190E-8
                & .280E-10
                \\
    \hline
                $97$
                & 320\,207
                & .97911
                & .02048
                & .390E-3
                & .624E-5
                & 0
                & 0
                & 0
                \\
    \multicolumn{2}{c}{Conjecture~\ref{conj}}
                & .97948
                & .02019
                & .312E-3
                & .429E-5
                & .553E-7
                & .684E-9
                & .823E-11
    \end{tabular}
\end{table}

\end{document}